\documentclass[12pt]{amsart}

\usepackage[margin=3cm]{geometry}

\usepackage{amsmath}

\usepackage{amsthm}
\newtheorem{thm}{Theorem}
\newtheorem{cor}[thm]{Corollary}
\newtheorem{lemma}{Lemma}[section]
\newtheorem{rmk}{Remark}[section]
\newtheorem{prop}{Proposition}[section]

\usepackage{amssymb, tikz, hyperref, bm, cleveref, enumitem}

\usepackage[foot]{amsaddr}

\begin{document}

\title{A Piecewise Smooth Fermi-Ulam Pingpong with Potential}
\author{Jing Zhou}
\address{Department of Mathematics, University of Maryland, College Park, MD 20742}
\email{jingzhou@math.umd.edu}
\date{\today}
\thanks{The author is much obliged to Dmitry Dolgopyat for proposing the problem and generously providing advice. The author also wants to thank Hong-Kun Zhang for her help and suggestions. This research was partially supported by the Patrick and Marguerite Sung Fellowship in Mathematics at University of Maryland and the NSF grant DMS 1665046.}
\maketitle

\begin{abstract}
    In this paper we study a Fermi-Ulam model where a pingpong bounces elastically against a periodically oscillating platform in a gravity field. We assume that the platform motion $f(t)$ is piecewise $C^3$ with a singularity $\dot{f}(0+)\ne\dot{f}(1-)$. If the second derivative of the platform motion is either always positive $\ddot{f}(t)>0$ or always $\ddot{f}(t)<-g$ where $g$ is the gravitational constant, then the escaping orbits constitute a null set and the system is recurrent. However, under these assumptions, escaping orbits coexist with bounded orbits at arbitrarily high energy level.
\end{abstract}

\section{Introduction}
There has been an extensive study on the Fermi-Ulam pingpong models since Fermi \cite{Fermi} and Ulam \cite{Ulam} proposed the bouncing ball mechanism as an explanation for the existence of high energy particles in the cosmic rays. The original Fermi-Ulam model describes a point particle bouncing elastically between two infinitely heavy wall, one fixed and the other oscillating periodically \cite{Ulam}. Ulam conjectured \cite{Ulam}, based on his numerical experiment with piecewise linearly oscillating wall, the existence of escaping orbits, i.e. orbits whose energy grows to infinity in time. In addition, bounded orbits (i.e. those whose energy always stays bounded) and oscillatory obits (i.e. those whose energy has a finite liminf but infinite limsup) might also exist in Fermi-Ulam models and various attempts have been made to examine the existence and prevalence of each of these three types of orbits (the author refers to \cite{Dol-FA,GRKTChaos,LiLi} for surveys).\\
Later KAM theory has negated the existence of accelerating orbits with sufficiently smooth wall motions as the prevalence of invariant curves prevents energy diffusion \cite{LaLe,Pu83,Pu94}. In nonsmooth cases, Zharnitsky \cite{Zhar} found linearly escaping orbits in a piecewise linear model. In a piecewise smooth model with one singularity de Simoi and Dolgopyat \cite{deSD} showed that there exists a parameter determining whether the linear part of the limiting system at infinity is elliptic or hyperbolic (i.e. whether the absolute value of trace is less or greater than 2), and that bounded orbits co-exist with escaping ones in elliptic regimes while escaping orbits have zero measure but full Hausdorf dimension in hyperbolic regimes.\\
When background potential is introduced, Arnold and Zharnitsky \cite{AZ} found unbounded orbits in a system with switching potentials. If the fixed wall is removed and gravity is present, Pustylnikov \cite{Pu77} showed that there exists an open set of wall motions in the space of analytic periodic functions admitting analytic extension to a fixed strip $|\Im t|<\varepsilon$ which produce infinite measure of escaping orbits. In a Duffing equation with a polynomial potential which possesses one discontinuity, Levi and You \cite{LY} proved the existence of oscillatory orbits. Ortega provided conditions for 
the existence of escaping orbits in piecewise linear oscillators \cite{Or99,Or02}. For intermediate cases where the potential takes the form $U=x^{\alpha}$ and the wall motion is sin-type, Dolgopyat \cite{Dol08} proved that the escaping orbits do not exist for $\alpha>1,\alpha\ne2$ and constitute a null set for $\alpha<1/3$, while de Simoi \cite{DS09} showed in the same setting that the escaping orbits possess full Hausdorff dimension for $\alpha<1$.\\

In this paper we study a Fermi-Ulam pingpong model with a potential. The model describes that a point mass is bouncing elastically against a moving wall in the gravity field. The motion (height) of the wall is a piecewise smooth periodic function $f(t)$ and the gravitational constant is given by $g$.

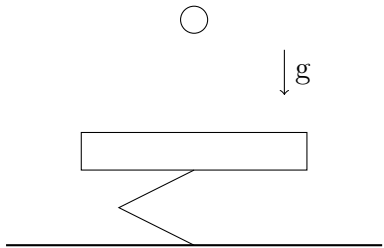
\begin{figure}[h!]
 \centering
   \begin{tikzpicture}
     \draw[thick] (-1,0) -- (4,0);
     \draw (0,1) rectangle (3,1.5);
     \draw (1.5,0) -- (0.5,0.5) -- (1.5,1);
     \draw (1.5,3) circle (0.18cm);
     \draw[->] (2.7,2.6) -- (2.7,2) node[anchor=south west] {g};
   \end{tikzpicture}
  \caption{Bouncing Pingpong in Gravity Field}\label{fig1}
\end{figure}

We are interested in the case when the motion $f(t)$ of the wall is 1-periodic and piecewise $C^3$, i.e $f(t+1)=f(t)$, $f\in C^3(0,1)$ and $\dot{f}(0+)\ne\dot{f}(1-)$. We record the time $t$ of each collision and the velocity $v$ immediately after each collision. We exclude from our discussion the singular collisions at integer times, which form a null set in the $(t,v)$-phase cylinder. We investigate the dynamics of the model by looking at the collision map $F$, which sends one collision $(t,v)$ to the next one $(\bar{t},\bar{v})$. Our main result is that if the second derivative of the wall motion behaves, i.e. the second derivative is either always positive $\ddot{f}(t)>0$ or always less than the negative of the gravitational constant $\ddot{f}(t)<-g$, then the escaping orbits have zero measure and $F$ is recurrent. 
We also show that under these assumptions, escaping and bounded orbits exist at arbitrarily high energy level.

\section{Main Results}

In this section we state the main results of this paper.

We denote the second derivative of the wall motion as $k(t) = \ddot{f}(t)$. The collision map $F$ preserves an absolutely continuous measure $\mu=wdtdv$ where $w=v-\dot{f}$ is the relative velocity after collision (c.f. Section 3.1).

For large velocities, the dynamics can be approximated by 
\[
F(t,v) = F_{\infty}(t,v) + \mathcal{O}\left(\frac{1}{v}\right)
\]
where $$F_{\infty}(t,v) = \left( t+\frac{2v}{g}, v+2\dot{f}\left(t+\frac{2v}{g}\right) \right).$$

It is easy to verify that the limit map $F_{\infty}$ is area-preserving and it covers a map $\tilde{F}_{\infty}$ on the torus $\mathbb{R}/\mathbb{Z}\times\mathbb{R}/g\mathbb{Z}$
\begin{equation*}
    \tilde{F}_{\infty} :
    \begin{cases}
    \tilde{t}_1 = \tilde{t}_0+\frac{2\tilde{v}_0}{g}\\
    \tilde{v}_1 = \tilde{v}_0+2\dot{f}(\tilde{t}_1)
    \end{cases}
\end{equation*}
where $\tilde{t} = t \mbox{ (mod 1)}$, $\tilde{v} = v \mbox{ (mod g})$.\\

If the second derivative $\ddot{f}$ of the wall motion is either always positive or always less than $-g$, then the limit map $\tilde{F}_{\infty}$ is ergodic: 

\begin{thm}\label{poserg}
Suppose that $\ddot{f}(t) > 0$ for any $t>0$. Then the map $\tilde{F}_{\infty}$ is ergodic.
\end{thm}

\begin{thm}\label{negerg}
Suppose that for any $t>0$, $\ddot{f}(t) <-g$ where $g$ is the gravitational constant. Then the map $\tilde{F}_{\infty}$ is ergodic.
\end{thm}

We shall call such wall motions \emph{admissible} if either for all $t$, $\ddot{f}>0$ or for all $t$, $\ddot{f}<-g$.

\begin{rmk}
  The map $\tilde{F}_{\infty}$ might not be ergodic if the assumptions in Theorem \ref{poserg} and Theorem \ref{negerg} fail. For example, when the wall motion is analytic 
  Pustylnikov \cite{Pu77} found an KAM island for the limit map for an open set of analytic periodic wall motions which admit analytic extension to a strip $|\Im t|<\varepsilon$. Note that for analytic
  motions $\int_0^1 \ddot{f}(t) dt=0$ so analytic motions are not admissible.
\end{rmk}

Besides the above ergodic properties, we also obtain stronger statistical properties of $\tilde{F}_{\infty}$ under the same assumptions.

For every $x,y \in \mathbb{T}$, we define their \textit{forward separation time} $s_+(x,y)$ to be the smallest nonnegative integer $n$ such that $x,y$ belongs to distinct continuity components of $\tilde{F}_{\infty}^n$. We can define similarly their \textit{backward separation time} $s_-(x,y)$ for the inverse iterates.\\
A function $\varphi:\mathbb{T} \to \mathbb{R}$ is said to be \textit{dynamically H\"older continuous} if there exist $\vartheta = \vartheta(\tilde{F}_{\infty}) \le 1$ such that
\[
|\varphi|_{\vartheta}^+ := \sup \left\{ \frac{|\varphi(x) - \varphi(y)|}{\vartheta^{s_+(x,y)}}: \hbox{$x \ne y$ on the same unstable manifold} \right\} <\infty,
\]
and that
\[
|\varphi|_{\vartheta}^- := \sup \left\{ \frac{\varphi(x) - \varphi(y)|}{\vartheta^{s_-(x,y)}}: \hbox{$x \ne y$ on the same stable manifold} \right\} <\infty .
\]
\begin{thm}[Exponential Decay of Correlations]\label{exp}
Suppose that the wall motion is admissible. Then the map $\tilde{F}_{\infty}$ enjoys exponential decay of correlations for dynamically H\"older continuous observables: $\exists \ b>0$ such that for any pair of dynamically H\"older continuous observables $\varphi, \phi$, $\exists \ C_{\varphi, \phi}$ such that 
\[
\bigg| \int_{\mathbb{T}} (\varphi \circ  \tilde{F}_{\infty}^n) \phi d\tilde{\mu} - \int_{\mathbb{T}}\varphi d\tilde{\mu}  \int_{\mathbb{T}} \phi d\tilde{\mu} \bigg| 
\le C_{\varphi, \phi} e^{-bn}, \ \ n\in\mathbb{N}.
\]
\end{thm}

We observe that for any dynamically H\"older observable $\varphi$, the following quantity is finite due to Theorem \ref{exp}
$$\sigma_{\varphi}^2:=\sum_{n=-\infty}^{\infty}\int_{\mathbb{T}}\varphi\cdot(\varphi\circ\tilde{F}_{\infty}^n) d\tilde{\mu}<\infty.$$
\begin{thm}[CLT]\label{clt}
Suppose that the wall motion is admissible. Then the map $\tilde{F}_{\infty}$ satisfies central limit theorem for dynamically H\"older observables, i.e. 
$$\frac{1}{\sqrt{n}}\sum_{i=0}^{n-1}\varphi\circ\tilde{F}_{\infty}^i \stackrel{dist}{\rightharpoonup}\mathcal{N}(0,\sigma_{\varphi}^2)$$
where $\varphi$ is dynamically H\"older with zero average $\int_{\mathbb{T}}\varphi d\tilde{\mu}=0$.
\end{thm}

As for the original system, under the assumptions in Theorem \ref{poserg} or Theorem \ref{negerg} the escaping orbit of the collision map $F$ constitute a null set:
\begin{thm}[Null Escaping Set]\label{null}
Suppose that the wall motion is admissible. Then the set $E$ of escaping orbits of $F$ has zero measure.
\end{thm}

It turns out that the escaping set is exactly the dissipative part of the system and consequently under the assumptions in Theorem \ref{poserg} or \ref{negerg} the system is recurrent:
\begin{cor}[Recurrence]\label{rec}
Suppose that the wall motion is admissible. Then $F$ is recurrent, i.e. almost every orbit comes arbitrarily close to its initial point.
\end{cor}

However, under the assumptions in Theorem \ref{poserg} or Theorem \ref{negerg}, escaping and bounded orbits still exist:
\begin{thm}\label{espbdd}
Suppose that $f(t)$ is admissible. Then $F$ possesses escaping and bounded orbits with arbitrarily high energy.
\end{thm}

Moreover, $F$ satisfies the following \emph{global global mixing} property for \emph{global functions}. We say a function $\Phi$ is \emph{global} if it is bounded, uniformly continuous and has a finite average $\bar{\Phi}$ in the following sense: for any $\epsilon$ there exists $N$ so large that for each square $V$ of size $\mu(V)>N$ we have $$\left\vert\frac{1}{\mu(V)} \int_V \Phi d\mu-\bar{\Phi}\right\vert \le \epsilon.$$ We denote by $\mathbb{G}_U$ the space of all such global functions.
\begin{thm}[Global Global Mixing]\label{ggm}
Suppose that the wall motion is admissible. Then $F$ is global global mixing with respect to $\mathbb{G}_U$, i.e. for any $\Phi_1,\Phi_2\in \mathbb{G}_U$, the following holds 
\begin{align*}
  \lim_{n\to\infty}\limsup_{\mu(V)\to\infty}\frac{1}{\mu(V)}\int_V \Phi_1\cdot(\Phi_2\circ F^n) d\mu &=\\
  \lim_{n\to\infty}\liminf_{\mu(V)\to\infty}\frac{1}{\mu(V)}\int_V \Phi_1\cdot(\Phi_2\circ F^n) d\mu &= \bar{\Phi}_1 \bar{\Phi}_2.
\end{align*}
\end{thm}

\section{Preliminaries}

In this section, we discuss the collision map. The study of the collision map relies substantially on the behavior of its limiting map, i.e. the approximated collision map for large velocities. We also discuss the singularity lines/curves of the limit map as they will play a very important role in the proofs later.

\subsection{The Collision Map}
We denote $s_n = t_{n+1} - t_n$ the flight time between two consecutive collisions.

Two consecutive collisions satisfy the following equations
\begin{equation}\label{collision}
 \left \{
   \begin{aligned}
    &-(v_n - g s_n - \dot{f}(t_{n+1})) = v_{n+1} - \dot{f}(t_{n+1})\\
    &f(t_n) + v_n s_n - \frac{1}{2}g s_n^2 = f(t_{n+1}).
   \end{aligned} \right.
\end{equation}

We compute the derivative of the collision map $F$ by differentiating these equations \cref{collision} 
\[
dF = 
  \begin{pmatrix}
    1+\frac{\dot{f}(t_n) - \dot{f}(t_{n+1})}{w_{n+1}}  & \frac{s_n}{w_{n+1}}\\
    2\ddot{f}(t_{n+1})+(2\ddot{f}(t_{n+1})+g) \frac{\dot{f}(t_n) - \dot{f}(t_{n+1})}{w_{n+1}} & (2\ddot{f}(t_{n+1})+g)\frac{s_n}{w_{n+1}} -1
   \end{pmatrix} . 
\]

We observe that $\det dF = \frac{w_n}{w_{n+1}}$ and hence $F$ preserves the measure $\mu=wdtdv$ on the phase cylinder.

\subsection{The Limit Map}

If we only consider collisions with large velocities, the dynamics can be approximated by 
\[
F(t,v) = F_{\infty}(t,v) + \mathcal{O}\left(\frac{1}{v}\right)
\]
where $$F_{\infty}(t,v) = \left( t+\frac{2v}{g}, v+2\dot{f}\left(t+\frac{2v}{g}\right)\right).$$

Denote as $(t_1,v_1) = F_{\infty}(t_0,v_0)$, then $$t_1 = t_0+\frac{2v_0}{g},\quad v_1 = v_0+2\dot{f}(t_1).$$
As mentioned in Section 2, the limit map $F_{\infty}$ covers a map $\tilde{F}_{\infty}$ on the torus $\mathbb{T}=\mathbb{R}/\mathbb{Z}\times\mathbb{R}/g\mathbb{Z}$
\[
\tilde{t}_1 = \tilde{t}_0+\frac{2\tilde{v}_0}{g}, \ \tilde{v}_1 = \tilde{v}_0+2\dot{f}(\tilde{t}_1)
\]
where $\tilde{t} = t \mbox{ (mod 1)}$, $\tilde{v} = v \mbox{ (mod g})$.\\

Denote as $k(t) = \ddot{f}(t)$. The dynamics of $\tilde{F}_{\infty}$ can be decomposed as 
\[
\tilde{t}_1 = \tilde{t}_0+\frac{2\tilde{v}_0}{g} \mbox{ (mod 1)}, \ \tilde{v}_0 = \tilde{v}_0,
\]
and 
\[
\tilde{t}_1 = \tilde{t}_1, \ \tilde{v}_1 =  \tilde{v}_0+2\dot{f}(\tilde{t}_1)
\]
hence the derivative of $\tilde{F}_{\infty}$ at $(\tilde{t}_0, \tilde{v}_0)$ is 
\[
d_{(\tilde{t}_0, \tilde{v}_0)}\tilde{F}_{\infty} = 
  \begin{pmatrix}
    1 & \frac{2}{g}\\
    2k_1 & \frac{4k_1}{g}+1
   \end{pmatrix}.
\]

We observe that $\det d\tilde{F}_{\infty} = 1$, so $\tilde{F}_{\infty}$ preserves the Lebesgue measure $\tilde{\mu}=d\tilde{t}d\tilde{v}$ on the torus.

\subsection{The Singularity Lines of the Limit Map}
A singularity occurs when the ball collides with wall at the singularities of the wall motion, i.e. $t \in \mathbb{N}$, hence the singularity line $\mathcal{S}^+$ of $\tilde{F}_{\infty}$ consists of the points whose next collisions happen at integer times, i.e.
\[
\mathcal{S}^+=\{\tilde{t}_1 = 0\}=\left\{\tilde{t}_0+\frac{2\tilde{v}_0}{g}\equiv_1 0\right\}. 
\]
Similarly, the singularity line $\mathcal{S}^-$ of $\tilde{F}_{\infty}^{-1}$ consists of the points whose the preimages land on integer times, i.e.
\[
\mathcal{S}^-=\{\tilde{t}_{-1} = 0\}=\left\{\tilde{t}_0+\frac{4}{g}\dot{f}(\tilde{t}_0)-\frac{2\tilde{v}_0}{g}\equiv_1 0\right\}.
\]

We observe that $\mathcal{S}^{\pm}$ consists of finitely many line/curve segments.

\section{Ergodicity of the Limit Map}
In this section we establish the ergodicity of the limit map $\tilde{F}_{\infty}$ under the assumptions in Theorem \ref{poserg} and Theorem \ref{negerg}. We use the result by Liverani and Wojtkowski in \cite{LW} where they proved ergodicity for a large class of Hamiltonian systems with invariant cones. We first describe the class of symplectic maps considered in \cite{LW} and then 
prove that $\tilde{F}_{\infty}$ satisfies the conditions of \cite{LW}.

Those conditions involve strictly invariant cones and the least coefficient of expansion, which are defined as follows.\\
For a point $p\in X$, we say $p$ possess \emph{strictly monotone cones} if there exist cone fields $\mathcal{C}(p)$ and its complementary $\mathcal{C}'(p)$ such that $d_pT$ preserves strictly the cone $\mathcal{C}(p)$ and that $d_pT^{-1}$ preserves strictly its complementary cone $\mathcal{C}'(p)$.\\
For $p\in X$ with strictly monotone cones $\mathcal{C}(p),\mathcal{C}'(p)$, there is an associated defining quadratic form $\mathcal{Q}_p$, i.e. $$\mathcal{C}(p)=\{\omega\in T_p X:\mathcal{Q}_p(\omega)\ge0\}.$$ The coefficient $\beta$ of expansion at $\omega\in T_p X$ is computed as $$\beta(\omega,d_p T)=\sqrt{\frac{\mathcal{Q}_p(d_p T\omega)}{\mathcal{Q}_p(\omega)}}$$ and the \emph{least coefficient $\sigma$ of expansion} is defined as $$\sigma(d_p T)=\inf_{\omega\in int\mathcal{C}(p)}\beta(\omega,d_p T)$$

Now we list here the six conditions of \cite{LW} in two dimensional case.

\begin{enumerate}
   \item The phase space $X$ is a finite disjoint union of compact subsets of a linear symplectic space $\mathbb{R}^2$ with dense and connected interior and \emph{regular} boundaries, i.e. they are finite unions of curves which intersect each other at at most finitely many points.
   
   \item For every $n \ge 1$ the singularity sets $\mathcal{S}_n^+$ and $\mathcal{S}_n^-$ of $T^n$ and $T^{-n}$ respectively are regular.
   
   \item Almost every point $p \in X$ possesses strictly monotone cones $\mathcal{C}(p)$ and its complementary $\mathcal{C}'(p)$.
   
   \item The singularity sets $\mathcal{S}^+$ and $\mathcal{S}^-$ are \textit{properly aligned}, i.e. the tangent line of $\mathcal{S}^-$ at any $p \in \mathcal{S}^-$ is contained strictly in the cone $\mathcal{C}(p)$ and the tangent line of $\mathcal{S}^+$ at any $p \in \mathcal{S}^+$ is contained strictly in the complementary cone $\mathcal{C}'(p)$. In fact, it is sufficient to assume that there exists $N$ such that $T^N \mathcal{S}^-$ and $T^{-N} \mathcal{S}^+$ are properly aligned.
   
   \item \emph{Noncontraction}: There is a constant $a\in(0,1]$ such that for every $n \ge 1$ and for every $p \in X \backslash \mathcal{S}_n^+$ 
\[
|| d_p T^n v || \ge a ||v||
\]
for every vector $v \in \mathcal{C}(p)$.

   \item \emph{Sinai-Chernov Ansatz}: For almost every $p \in \mathcal{S}^-$ with respect to the measure $\mu_{\mathcal{S}}$, its least coefficient of expansion satisfies 
   \[
   \lim_{n \to \infty} \sigma (d_p T^n) = \infty.
   \]
\end{enumerate}

We note from \cite{LW} that $\sigma$ is supermultiplicative, i.e. $\sigma(L_1L_2) \ge \sigma(L_1) \sigma(L_2)$, and that if the coordinates are such that the cone $\mathcal{C}(p)$ is the positive cone and $d_pT$ takes the form 
\[d_p T=
   \begin{pmatrix}
        A & B\\
        C & D
    \end{pmatrix}
\]
then $\sigma$ can be computed as $\sigma(L) = \sqrt{1+t} + \sqrt{t}$ where $t = BC$.

Liverani and Wojtkowski have proved local ergodicity for symplectic maps satisfying the above six conditions:

\begin{thm}[\cite{LW}]\label{thmlw}
Suppose $(T,X)$ satisfies the above conditions. For any $n \ge 1$ and for any $p \in X \backslash \mathcal{S}_n^+$ such that $\sigma (d_p T^n)>3$ there is a neighborhood of $p$ which is contained in one ergodic component of $T$.
\end{thm}

Now we prove Theorem \ref{poserg}:

\begin{proof}[Proof of Theorem 1]

Suppose that $\ddot{f}>0$.\\

First we prove local ergodicity by verifying the above six conditions for $\tilde{F}_{\infty}$.\\

The singularity lines $\mathcal{S}^{\pm}$ are finite unions of short lines/curves and they cut our phase space $X$, which is a torus, into finitely many pieces.\\

The strict monotonicity follows easily from the fact that $d\tilde{F}_{\infty}$ is positive when $\ddot{f}>0$, hence $d_p\tilde{F}_{\infty}$ preserves strictly the positive cone $\mathcal{C}^+(p) = \{ \delta \tilde{t} \delta \tilde{v} \ge 0 \}$ and $d_p\tilde{F}_{\infty}^{-1}$ preserves strictly the complementary negative cone $\mathcal{C}^-(p) = \{ \delta \tilde{t} \delta \tilde{v} \le 0 \}$.\\

It is straightforward from the previous discussion that $\mathcal{S}^{\pm}$ are properly aligned since the slope of tangent line to $\{ \tilde{t}_1 = 0 \}$ at $(\tilde{t}_0,\tilde{v}_0)$ is $\displaystyle -\frac{g}{2}<0$, and the slope of the tangent line to $\{ \tilde{t}_{-1} = 0 \}$ is $\displaystyle \frac{g}{2}\left(1+\frac{4k_0}{g}\right)>0$.\\

Next we verify the noncontraction property.\\
For any nonsingular point $p = (\tilde{t}, \tilde{v})$, any vector $\bm{v} = (\delta \tilde{t}, \delta \tilde{v}) \in \mathcal{C}^+(p)$ with $\delta \tilde{t} \delta \tilde{v} \ge 0$, 
\begin{equation}\label{expansionrate}
  \begin{split}
    || d_p\tilde{F}_{\infty} \bm{v} ||^2 
    & = (1+4k_1^2) \delta \tilde{t}^2 + \left(\frac{4}{g^2} + \left(\frac{4k_1}{g}+1\right)^2\right) \delta \tilde{v}^2 + \left(\frac{4}{g} + 4k_1\left(\frac{4k_1}{g}+1\right)\right)\delta \tilde{t} \delta \tilde{v}\\
    & \ge \delta \tilde{t}^2 + \delta \tilde{v}^2 = || \bm{v} ||^2
  \end{split}
\end{equation}
thus the noncontraction property follows for $a=1$.\\

Now we verify the Sinai-Chernov Ansatz.\\
We denote as $\mathcal{S}_0 = \{ \tilde{t}_0 = 0 \}$. For any point $p \in \mathcal{S}^- \backslash \cup_{n \ge 0} \mathcal{S}_n$, which excludes a $\tilde{\mu}_{\mathcal{S}^-}$-null set since $\mathcal{S}^-$ intersect each $\mathcal{S}_n$ at at most finitely many points,
\[
\sigma(d_p\tilde{F}_{\infty}) = \sqrt{1+\frac{4k_1}{g}} + \sqrt{\frac{4k_1}{g}} \ge \sqrt{1+\frac{4k_{\min}}{g}} + \sqrt{\frac{4k_{min}}{g}} >1,
\]
where $\displaystyle k_{\min}=\min_{t\in(0,1)}\ddot{f}(t)>0$ by our assumption, then the supermultiplicativity of $\sigma$ implies that $\displaystyle \lim_{n \to \infty} \sigma (d_p \tilde{F}_{\infty}^n) = \infty$.\\

Finally, it remains to check that the singularity sets $\mathcal{S}_n^-$ and $\mathcal{S}_n^+$ of $\tilde{F}_{\infty}^n$ and $\tilde{F}_{\infty}^{-n}$ respectively are regular. We claim that for every $n>0$, $\mathcal{S}_n^-$ ($\mathcal{S}_n^+$) is a finite union of increasing (decreasing) curves, i.e. curves with bounded positive (negative) slope. This can be proved by an inductive argument. Firstly the claim holds for $n=1$ as already shown before. Now suppose $\mathcal{S}_n^-$ is a finite union of increasing curves. Since $\mathcal{S}_{n+1}^- = \mathcal{S}_n^- \cup \tilde{F}_{\infty} \mathcal{S}_n^-$ and $d\tilde{F}_{\infty}$ is a positive matrix and the second derivative $\ddot{f}$ is bounded, $\mathcal{S}_{n+1}^-$ is a finite union of increasing curves. The claim for $\mathcal{S}_n^+$ can be proved similarly. \\

Observe that $k(t)=\ddot{f}(t)>0$ is uniformly bounded, hence there exists $N>0$ such that $\sigma(d_p\tilde{F}_{\infty}^N) > 3$ for any $p \notin \mathcal{S}^+_N$. Therefore we have obtained local ergodicity for $\tilde{F}_{\infty}$ by Theorem \ref{thmlw}.\\

Now we argue for global ergodicity by contradiction.\\
Suppose that there exists some nontrivial ergodic component $M$ of $\tilde{F}_{\infty}$, then its boundary $\partial M$ lies on $\mathcal{S}^+_N$. But 
$$\mathcal{S}^+_N = \bigcup_{n=0}^{N-1} \tilde{F}_{\infty}^{-n} \mathcal{S}^+$$
hence there exists a smallest integer $N_0 \ge 1$ such that $\partial M \in \mathcal{S}^+_{N_0}$.\\
Observe that $\tilde{F}_{\infty}(\partial M) = \partial M$ by the invariance of $M$. However, $\tilde{F}_{\infty}(\mathcal{S}^+_{N_0}) = \mathcal{S}^+_{N_0 -1} \cup \mathcal{S}_0$, which contradicts the minimality of $N_0$. Note that although $\tilde{F}_{\infty}$ is multivalued at $\mathcal{S}^+$, we have $\tilde{F}_{\infty} \mathcal{S}^+ = \mathcal{S}_0$ anyway.\\

Therefore we conclude that there cannot be any nontrivial ergodic component and hence $\tilde{F}_{\infty}$ is ergodic.
\end{proof}

The proof of Theorem \ref{negerg} follows a similar strategy. 
The main difficulty arises from finding invariant cones as the derivative matrix is no longer positive. However we still manage to construct invariant cones out of the ``eigenvectors" of the derivative matrix.

\begin{proof}[Proof of Theorem 2]
Suppose $\ddot{f}<-g$.\\
First of all, we recall that the derivative of $\tilde{F}_{\infty}$ at $(\tilde{t}_0, \tilde{v}_0)$ is 
\[
d_{(\tilde{t}_0, \tilde{v}_0)}\tilde{F}_{\infty} = 
  \begin{pmatrix}
    1 & \frac{2}{g}\\
    2k_1 & \frac{4k_1}{g}+1
   \end{pmatrix}.
\]

Now we consider the following two cones 
\[
\mathcal{C}^u(\tilde{t}_0,\tilde{v}_0)=\left\{\frac{\delta v}{\delta t}\le k_0\right\},\quad 
\mathcal{C}^s(\tilde{t}_0,\tilde{v}_0)=\left\{\frac{\delta v}{\delta t}\ge k_0\right\}
\]
We verify that they are invariant under $d_{(\tilde{t}_0,\tilde{v}_0)}\tilde{F}_{\infty}$ and $d_{(\tilde{t}_0,\tilde{v}_0)}\tilde{F}_{\infty}^{-1}$ respectively.\\
For any $(\delta t,\delta v)\in\mathcal{C}^u(\tilde{t}_0,\tilde{v}_0)$,
\[
  \begin{pmatrix} \bar{\delta t} \\ \bar{\delta v} \end{pmatrix} = 
  \begin{pmatrix} 1&\frac{2}{g} \\ 2k_1&\frac{4k_1}{g}+1 \end{pmatrix}
  \begin{pmatrix} \delta t \\ \delta v \end{pmatrix}=
  \begin{pmatrix} \delta t+\frac{2}{g}\delta v \\ 2k_1(\delta t+\frac{2}{g}\delta v)+\delta v \end{pmatrix}
\]
thus 
\begin{equation}\label{unstablecone}
  \begin{split}
    \frac{\bar{\delta v}}{\bar{\delta t}} 
    &= 2k_1 + \frac{\frac{\delta v}{\delta t}}{1+\frac{2}{g}\frac{\delta v}{\delta t}}\\
    &\le 2k_1+\frac{k_0}{1+\frac{2k_0}{g}}\\
    &< 2k_1+g < k_1
  \end{split}
\end{equation}
where the first inequality follows from $\frac{\delta v}{\delta t}\le k_0$ and the last two inequalities from $\ddot{f}<-g$. Thus $(\bar{\delta t},\bar{\delta v})\in\mathcal{C}^u(\tilde{t}_1,\tilde{v}_1)$.\\
For any $(\delta t,\delta v)\in\mathcal{C}^s(\tilde{t}_0,\tilde{v}_0)$
\[
  \begin{pmatrix} \tilde{\delta t} \\ \tilde{\delta v} \end{pmatrix} = 
  \begin{pmatrix} \frac{4k_0}{g}+1&-\frac{2}{g} \\ -2k_0&1 \end{pmatrix}
  \begin{pmatrix} \delta t \\ \delta v \end{pmatrix}=
  \begin{pmatrix} \delta t-\frac{2}{g}(-2k_0\delta t+\delta v)\\ -2k_0\delta t+\delta v \end{pmatrix},
\]
thus 
\begin{equation}\label{stablecone}
  \begin{split}
    \frac{\tilde{\delta v}}{\tilde{\delta t}} 
    &= -\frac{g}{2}+\frac{g}{2}\frac{1}{1+\frac{2}{g}(2k_0-\frac{\delta v}{\delta t})}\\
    &> -\frac{g}{2}+\frac{g}{2}\frac{1}{1+\frac{2k_0}{g}}\\
    &> -g > k_{-1}
  \end{split}
\end{equation}
where the first inequality follows from $\frac{\delta v}{\delta t}\ge k_0$ and the last two inequalities from $\ddot{f}<-g$. Thus $(\tilde{\delta t},\tilde{\delta v})\in\mathcal{C}^u(\tilde{t}_{-1},\tilde{v}_{-1})$.\\

If we can verify that $\mathcal{S}^{\pm}$ are properly aligned, then the regularity of the singularity curves $\mathcal{S}_n^{\pm}$ follows automatically from the strict invariance of the cones $\mathcal{C}^{u/s}$ since $\mathcal{S}_n^{\pm}$ consist of finitely many transverse short curves.\\
We claim that $\mathcal{S}^{\pm}$ are properly aligned. Indeed, the slope of tangent line to $\{ \tilde{t}_1 = 0 \}$ at $(\tilde{t}_0,\tilde{v}_0)$ is $\displaystyle -\frac{g}{2}>k_0$, which is properly contained in $\mathcal{C}^s$. Also, the slope of the tangent line to $\{ \tilde{t}_{-1} = 0 \}$ is $\displaystyle 2k_0+\frac{g}{2}<k_0$, which is properly contained in $\mathcal{C}^u$.\\

The noncontraction property still holds with $a=1$.\\

The unstable cone $\mathcal{C}^u$ is not canonical, i.e. it is not the positive cone, hence we need to switch to the new basis $((0,1),(1,k_0))$ and $d_p\tilde{F}_{\infty}$ takes the form 
\[
\begin{pmatrix}
  \frac{2k_1}{g}+1 & k_0+k_1+\frac{2k_0k_1}{g}\\
  2/g & 1+\frac{2k_0}{g}
\end{pmatrix}
\]
Then the Sinai-Chernov Ansatz holds since 
\begin{align*}
    \sigma(d_p\tilde{F}_{\infty}) 
    &= \sqrt{1+\frac{2}{g}\left(\frac{2k_0k_1}{g}+k_0+k_1\right)} + \sqrt{\frac{2}{g}\left(\frac{2k_0k_1}{g}+k_0+k_1\right)}\\
    &\ge \sqrt{1+\frac{2}{g}\left(\frac{2k_{\min}^2}{g}+2k_{\min}\right)} + \sqrt{\frac{2}{g}\left(\frac{2k_{\min}^2}{g}+2k_{\min}\right)}>1.
\end{align*}

By Theorem \ref{thmlw} the local ergodicity of $\tilde{F}_{\infty}$ for the case when $\ddot{f}<-g$. Finally, the global ergodicity can be obtained by a similar argument as in the proof of Theorem \ref{poserg}.
\end{proof}

\section{Recurrence of the Collision Map}
In this section, we prove Theorem \ref{null} and Theorem \ref{rec} as they are direct consequences of the ergodicity of the limit map $\tilde{F}_{\infty}$ on the torus.\\
The proof of Theorem \ref{null} uses a result from \cite{deSD}, which shows that for an asymptotically periodic map with an ergodic limiting map, if the energy change in the limiting map has zero average, then the escaping orbits of the original dynamics have zero measure.

We state this result for our case specifically. First we decompose the velocity $v$ into integer part and fractional part, i.e. there exists some $m\in\mathbb{Z}$ such that 
$$v=\tilde{v}+mg \text{ where } \tilde{v} \in [0, g). $$ 
Then we decompose the limit map $F_{\infty}$ on the cylinder into its projection $\tilde{F}_{\infty}$ on the torus and a map $\gamma$ on integers $\mathbb{Z}$, i.e. 
$$(\tilde{t}_1, \tilde{v}_1, m_1) = F_{\infty}(\tilde{t}_0, \tilde{v}_0, m_0)=(\tilde{F}_{\infty}(\tilde{t}_0,\tilde{v}_0),m_0+\gamma(\tilde{t}_0,\tilde{v}_0)).$$
\begin{lemma}[\cite{deSD}]\label{lemmadsd}
Suppose that $\tilde{F}_{\infty}$ is ergodic with respect to the measure $\tilde{\mu}=d\tilde{t}d\tilde{v}$ on the torus. If the energy change of $F_{\infty}$ has zero average, i.e. $\int_{\mathbb{T}} \gamma(\tilde{t}_0, \tilde{v}_0) d\tilde{\mu} = 0$, then the escaping set of $F$ has zero measure.
\end{lemma}

Now we prove Theorem \ref{null}.
\begin{proof}[Proof of Theorem \ref{null}]
If $f(t)$ is admissible, 
then, by Theorems \ref{poserg} and  \ref{negerg}, $\tilde{F}_{\infty}$ is ergodic. Thus by Lemma \ref{lemmadsd}, it suffices to check that the energy change $\gamma$ of $F_{\infty}$ has zero average.\\
With an abuse of notation, let us denote $(t_1, v_1) = F_{\infty}(t_0, v_0)$ and $(\tilde{t}_1, \tilde{v}_1) = \tilde{F}_{\infty}(\tilde{t}_0, \tilde{v}_0)$ respectively. Observe that $$\int \tilde{v}_0 d\tilde{\mu} = \int \tilde{v}_1 d\tilde{\mu}$$ since $\tilde{F}_{\infty}$ preserves the measure $\tilde{\mu}$. Thus $$\int \gamma d\tilde{\mu} = \frac{1}{g} \int (v_1 - v_0) d\tilde{\mu}.$$
But $v_1 - v_0 = 2\dot{f}(\tilde{t}_1)$, hence 
\[
\int (v_1 - v_0) d\tilde{\mu} = \int 2\dot{f}(\tilde{t}_1) d\tilde{\mu} =  \int 2\dot{f}(\tilde{t}_0) d\tilde{\mu} =0
\]
where the last two equalities follow from the fact that $\tilde{F}_{\infty}$ preserves $\tilde{\mu}$ and that $f$ is 1-periodic.
Therefore by Lemma \ref{lemmadsd} the escaping orbits of $F$ have zero measure.
\end{proof}

The set $E$ of escaping orbits is in fact the transient part of $F$, hence Theorem \ref{null} implies that $F$ is recurrent, in the spirit of \cite{Dol}:

\begin{proof}[Proof of Corollary \ref{rec}]
We claim that the set $E$ of escaping orbits is the transient component of the system. Hence Theorem \ref{null} implies Corollary~\ref{rec}.

Indeed, the complement of $E$ is $\cup E_N$ where
\[
E_N = \{ (t_0,v_0): \liminf v_n \le N \}.
\]
For any $N\in\mathbb{N}$, $E_N$ is invariant and all points in $E_N$ will visit the set $V_N=\{ v \le N+1 \}$.\\
Suppose $A$ is a subset of $E_N$ with finite measure. For any $x\in A$, we denote the first hitting time in $V_N$ as $r(x)=\min\{k\ge0: F^k x\in V_N\}$. Now for any $K\in\mathbb{N}$, we consider $$A_K:=\bigcup_{x\in A:r(x)\le K}F^{r(x)}x$$
To show recurrence in $A$, it suffices to prove that almost every point in $A_K$ visits itself infinitely often since if for $x\in A_K$ $F^nx\in A_K$ for some $n>K$, then there exists some $x'\in A$ such that $F^{n-r(x')}x=x'\in A$.\\
However $A_K\subseteq E_N\cap V_N$ by definition of $A_K$ and the invariance of $E_N$. All points in $E_N$ visit $V_N$, thus the first return map $P$ on $A_K$ is well-defined. Now our goal is achieved by applying Poincar\'e recurrence theorem to $(A_K,P)$.
\end{proof}

\section{Statistical Properties of the Limit Map}
In this section we prove Theorem \ref{exp}, Theorem \ref{clt} and Theorem \ref{ggm}. Throughout this section we assume the wall motion is admissible.

\subsection{Background.}
The proof of Theorem \ref{exp} uses a result of Chernov and Zhang in \cite{CZ} and the proof of Theorem \ref{clt} uses a result of Chernov in \cite{Ch}.
We first describe the class of hyperbolic symplectic maps considered in \cite{Ch,CZ} and then show that our map $\tilde{F}_{\infty}$ belongs to this class.

Let $T: M \to M$ be a $C^2$ diffeomorphism of a two dimensional Riemannian manifold $M$ with singularities $\mathcal{S}$. Suppose $T$ satisfies the following conditions: 

\begin{enumerate}
   \item \emph{Uniform hyperbolicity of $T$}. There exist two continuous families of unstable cones $\mathcal{C}_x^u$ and stable cones $\mathcal{C}_x^s$ in the tangent spaces $\mathcal{T}_x M$ for all $x\in M$, and $\exists$ a constant $\Lambda>1$ such that
   
        \begin{enumerate}
            \item $DT(\mathcal{C}_x^u) \subset \mathcal{C}_{Tx}^u$, and $DT(\mathcal{C}_x^s) \supset \mathcal{C}_{Tx}^s$ whenever $DT$ exists;
            \item $|| D_x T \rm{v} || \ge \Lambda || \rm{v} ||$ $\forall \rm{v} \in \mathcal{C}_x^u$, and $|| D_x T^{-1} \rm{v} || \ge \Lambda || \rm{v} ||$ $\forall \rm{v} \in \mathcal{C}_x^s$;
            \item The angle between $\mathcal{C}_x^u$ and $\mathcal{C}_x^s$ is uniformly bounded away from zero.
         \end{enumerate}

   \item \emph{Singularities $\mathcal{S}^{\pm}$ of $T$ and $T^{-1}$}. The singularities $\mathcal{S}^{\pm}$ have the following properties:
   
         \begin{enumerate}
             \item $T: M\backslash \mathcal{S}^+ \to M\backslash \mathcal{S}^-$ is a $C^2$ diffeomorphism;
             
             \item $\mathcal{S}_0 \cup \mathcal{S}^+$ is a finite or countable union of smooth compact curves in $M$;
             
             \item Curves in $\mathcal{S}_0$ are transverse to the stable and unstable cones. Every smooth curve in $\mathcal{S}^+$ ($\mathcal{S}^-$) is a stable (unstable) curve. Every curve in $\mathcal{S}^+$ terminates either inside another curve of $\mathcal{S}^+$ or on $\mathcal{S}_0$;
             
             \item $\exists \ \beta \in (0,1)$ and $c>0$ such that for any $x \in M \backslash \mathcal{S}^+$, $|| D_x T|| \le c d(x, \mathcal{S}^+)^{-\beta}$.
         \end{enumerate}

    \item \emph{Regularity of smooth unstable curves.} We assume there exists a $T$-invariant class of unstable curves $W$ such that
    
         \begin{enumerate}
              \item \emph{Bounded curvature.} The curvature of $W$ is uniformly bounded from above;
              
              \item \emph{Distortion control.} $\exists \ \gamma \in (0,1)$ and $C>1$ such that for any regular unstable curve $W$ and any $x,y \in W$
\[
\big| \log \mathcal{J}_W(x) - \log \mathcal{J}_W(y) \big| \le C d(x,y)^{\gamma},
\]
where $\mathcal{J}_W(x) = |D_x T|_W|$ denotes the Jacobian of $T$ at $x\in W$;

              \item \emph{Absolute continuity of the holonomy map.} Let $W_1, W_2$ be two regular unstable curves that are close to each other. We denote 
\[
W_i' = \{ x\in W_i: W^s(x) \cap W_{3-i} \ne \emptyset \}, \ \ i=1,2.
\]
The holonomy map $h: W_1' \to W_2'$ is defined by sliding along the stable manifold. We assume that $h_{\star}\mu_{W_1'}\ll\mu_{W_2'}$ and that for some constant $C$ and $\vartheta<1$
\[
\big| \log \mathcal{J}h(x) - \log \mathcal{J}h(y) \big| \le C \vartheta^{s_+(x,y)}, \quad x,y \in W_1'
\]
where $\mathcal{J}h$ is the Jacobian of $h$;
       \end{enumerate}

   \item \emph{SRB measure}. $\tilde{\mu}$ is an SRB measure, i.e. the induced measure $\tilde{\mu}_{W^u}$ on any unstable manifold $W^u$ is absolutely continuous with respect to $Leb_{W^u}$. We also assume that $\tilde{\mu}$ and is mixing.

   \item \emph{One-step expansion.} Let $\xi^n$ denote the partition of $M$ into connected components of $M\backslash \mathcal{S}_n^+$. Denote as $V_{\alpha}$ the connected component of $TW$ with index $\alpha \in M/\xi^1$ and $W_{\alpha} = T^{-1}V_{\alpha}$. $\exists \ q\in (0,1]$ such that 
\[
\liminf_{\delta \to 0} \sup_{W: |W| < \delta} \sum_{\alpha \in M/ \xi^1} \bigg( \frac{|W|}{|V_{\alpha}|} \bigg)^q \frac{|W_{\alpha}|}{|W|} < 1,
\]
where the supremum is taken over all unstable curves $W$.
\end{enumerate}

\begin{thm}[\cite{Ch,CZ}]\label{thmcz}
Under the assumptions 1-5 above, the system $(T,M)$ above enjoys exponential decay of correlations and central limit theorem for dynamically H\"older continuous observables.
\end{thm}

The verifications of Assumptions 1-5 are rather long. Moreover, their validity  is of independent importance themselves. So we first state intermediary lemmata in Section 6.2, then we prove, based on these lemmata, Theorem \ref{exp}, Theorem \ref{clt} and Theorem \ref{ggm} in Section 6.3, and finally we prove all the lemmata in Section 6.4.

\subsection{Intermediary Lemmas}
In this section we list the intermediate lemmata. Their proofs are presented in Section 6.4.

Suppose $W$ is an unstable curve, i.e. the tangent line of $W$ lies in the unstable cone $\mathcal{C}^u$, with bounded curvature. We assume without loss of generality that $W \cap \mathcal{S}^+ = \emptyset$. Let $\mathcal{J}_W(x) = |D_x T|_W|$ denote the Jacobian of $\tilde{F}_{\infty}$ at $x\in W$. We have the following enhanced distortion control:
\begin{lemma}[Distortion Control]\label{distortion}
Suppose $W$ is an unstable curve with bounded curvature. Then for any $x,y\in W$, there exists a constant $C$ which depends only on $\tilde{F}_{\infty}$ such that
\[
\big| \log \mathcal{J}_W(x) - \log \mathcal{J}_W(y) \big| \le C d(x,y).
\]
Furthermore, if $W \cap \mathcal{S}_N^- = \emptyset$, then for any $1 \le n \le N$ there exists a constant $C'$ which depends only on $\tilde{F}_{\infty}$ such that
\[
\big| \log \mathcal{J}_W \tilde{F}_{\infty}^{-n} (x) - \log \mathcal{J}_W \tilde{F}_{\infty}^{-n} (y) \big| \le C' |W|.
\]
\end{lemma}

In order to establish the $N_0$-step expansion, we need the following estimate on the speed of fragmentation of unstable curves:
\begin{lemma}[Complexity Bound]\label{complexity}
Suppose $z$ is a multiple point of $\mathcal{S}_n^+$. Pick a small neighborhood of $z$ and denote as $k_n(z)$ the number of sectors in the small neighborhood cut out by $\mathcal{S}_n^+$. Then $k_n(z) \le 6n$.
\end{lemma}

The linear complexity bound guarantees that a sufficiently short unstable curve $W$ can break into at most $6n$ connected components under $\tilde{F}_{\infty}^n$. Thus there exists $\delta_0$ so small that any unstable curve shorter than $\delta_0$ satisfies the following expansion estimate:
\begin{lemma}[$N_0$-Step Expansion]\label{nstep}
Suppose that $W$ is an unstable curve with length $|W|\le\delta_0$ and that $\{W_i^n\}_i$ are the connected components of the image $\tilde{F}_{\infty}^nW$. Denote as $\Lambda_{i,n}$ the minimum rate of expansion on each preimage $\tilde{F}_{\infty}^{-n}W_i^n$. Then 
$$\sum_i \frac{1}{\Lambda_{i,N_0}} < 1,$$
where $N_0$ is the smallest integer such that $\displaystyle \frac{6N_0}{\Lambda^{N_0}}<1$ and $\Lambda$ is the expansion rate of $\tilde{F}_{\infty}$.
\end{lemma}

Next, we suppose that $W$ and $\bar{W}$ are two unstable curves with bounded curvatures. We define the following holonomy map $h$ on $$W'=\{x\in W:W^s(x)\cap\bar{W}\ne\emptyset\}$$ by sliding along the stable manifold from $x\in W$ to $\bar{x}\in\bar{W}$. Then $h:W'\to\bar{W}'$ is absolutely continuous with well-behaving density:
\begin{lemma}[Absolute Continuity]\label{holonomy}
Suppose $W$ and $\bar{W}$ are two unstable curves with bounded curvatures. Then $h_{\star}\mu_{W}\ll\mu_{\bar{W}}$ and that for some constant $C$ and $\Theta<1$
\[
\big| \log \mathcal{J}h(x) - \log \mathcal{J}h(y) \big| \le C \Theta^{s_+(x,y)}, \quad x,y \in W'
\]
where $\mathcal{J}h$ is the Jacobian of $h$.
\end{lemma}

Finally we provide an estimate on the number of small unstable curves, which follows from Lemma 7 in \cite{CZ}.\\
Suppose $W$ is an unstable curve with length $|W|<\delta_0$. For any $x\in W$, we denote as $r_n(x)$ the distance from $x$ to the nearest boundary of the connected component of $\tilde{F}_{\infty}^n W$ containing $\tilde{F}_{\infty}^n x$.
\begin{lemma}[Growth Lemma]\label{growthlemma}
  Suppose $W$ is an unstable curve with length $|W|<\delta_0$. Then for any $\epsilon>0$, 
  $$\text{mes}_W\{r_{nN_0}(x)<\epsilon\}\le (\vartheta_1\Lambda^{N_0})^n \text{mes}_W
  \left\{r_0(x)<\frac{\epsilon}{\Lambda^{nN_0}}\right\}+C\epsilon|W|$$
  where $\vartheta_1=e^{C'\delta_0}\sum_i\frac{1}{\Lambda_{i,N_0}}<1$, $C'$ is the constant from Lemma \ref{distortion}, $N_0$ is the constant from Lemma \ref{nstep} and $\Lambda$ is the expansion rate of $\tilde{F}_{\infty}$.
\end{lemma}

\begin{rmk}
  We note that $\vartheta_1$ can be made less than 1 by choosing $\delta_0$ sufficiently small.
\end{rmk}

\subsection{Exponential Decay of Correlations, CLT and Global Global Mixing}
In this section we present the proof of Theorem \ref{exp}, Theorem \ref{clt} and Theorem \ref{ggm}, based on the lemmas from Section 6.2.\\

We start with the proof for exponentially decay of correlations and CLT.
\begin{proof}[Proof of Theorem \ref{exp} and Theorem \ref{clt}]
Firstly we establish the exponential decay and CLT for $\tilde{F}_{\infty}^{N_0}$ by checking the conditions in Theorem \ref{thmcz} for $\tilde{F}_{\infty}^{N_0}$ where $N_0$ is the number from Lemma \ref{nstep}. Note that we gain from Theorem \ref{thmcz} the exponential decay and CLT for $\tilde{F}_{\infty}^{N_0}$ rather than for $\tilde{F}_{\infty}$ because we can only obtain $N_0$-step expansion on $\tilde{F}_{\infty}$.\\
The proof for the case $\ddot{f}>0$ is very similar to that for $\ddot{f}<-g$, and thence we omit the latter. Although the positive/negative cones in the proof of Theorem \ref{poserg} are strictly invariant, we cannot use them here since we require a positive angle between the unstable and stable cones. Instead, we consider their images, i.e. 
\[
\mathcal{C}^u (\tilde{t}_0, \tilde{v}_0) = \left\{2k_0 \le \frac{\delta v}{\delta t} \le 2k_0 + \frac{g}{2}\right\}
\]
\[
\mathcal{C}^s (\tilde{t}_0, \tilde{v}_0) = \left\{-\frac{g}{2} \le \frac{\delta v}{\delta t} \le -\frac{2k_0}{\frac{4k_0}{g} +1}\right\}.
\]
It is easy to see that the angles between $\mathcal{C}^u$ and $\mathcal{C}^s$ are uniformly bounded away from zero since $k_0>0$ is bounded.
and that these cones are strictly invariant.


We now compute the expansion rate $\Lambda$.
With the same notations as above, for $(\delta t, \delta v) \in \mathcal{C}^u (\tilde{t}_0, \tilde{v}_0)$ it follows from (\ref{expansionrate}) that 
\[
\bar{\delta t}^2 + \bar{\delta v}^2 \ge \Lambda_1^2 (\delta t^2 + \delta v^2)
\]
where $\displaystyle \Lambda_1^2 = \min \left\{1+4k_{min},\frac{4}{g^2} + \left(1+\frac{4k_{min}}{g}\right)^2\right\}>1$.\\
Similarly for $(\delta t, \delta v) \in \mathcal{C}^s (\tilde{t}_0, \tilde{v}_0)$
\begin{align*}
\tilde{\delta t}^2 + \tilde{\delta v}^2 
&= \left(4k_1^2 + \left(1+\frac{4k_1}{g}\right)^2\right) \delta t^2 + \left(\frac{4}{g^2} + 1\right) \delta v^2 - 2\left(2k_1 + \frac{2}{g}\left(1+\frac{4k_1}{g}\right)\right) \delta t \delta v \\
&\ge \Lambda_2^2 (\delta t^2 + \delta v^2)
\end{align*}
where $\displaystyle \Lambda_2^2 = \min\left\{4k_{min}^2 + \left(1+\frac{4k_{min}}{g}\right)^2, \frac{4}{g^2} + 1\right\} >1$.\\
Take $\Lambda = \min \{ \Lambda_1, \Lambda_2 \}$, and this gives our expansion rate.\\

Next we examine the singularity curves $\mathcal{S}^{\pm}$.\\
Observe that $\tilde{F}_{\infty}$ is a $C^2$-diffeomorphism away from singularities if $f$ is piecewise $C^3$. And $\mathcal{S}_0 \cup \mathcal{S}^+$ is a finite union of smooth compact curves on the torus $\mathbb{T}$. $\mathcal{S}_0$ is transverse to $\mathcal{C}^u/\mathcal{C}^s$. Moreover the singularity curves are regular and properly aligned as shown in the proof of Theorem \ref{poserg}.\\
Assumption 2(d) is trivially satisfied since the norm of the derivative $d\tilde{F}_{\infty}$ is bounded and our phase space is compact.\\

As for Assumption 3, we have already obtained distortion control in Lemma \ref{distortion} and absolute continuity of holonomy map in Lemma \ref{holonomy}. We note that by (\ref{boundedcurvature}) the curvature of an unstable curve remains bounded after iterations.\\

Next, the invariant measure $\tilde{\mu} =d\tilde{t}d\tilde{v}$ is apparently an SRB measure.\\
Note that $\tilde{F}_{\infty}^n$ is ergodic with respect to $\tilde{\mu}$ for any $n>0$, since $\tilde{F}_\infty^n$ also satisfies the conditions of Theorem \ref{thmlw}. Now
the results of \cite{Pes} imply that $\tilde{\mu}$ is mixing (even Bernoulli).\\

Finally, since we already establish $N_0$-step expansion from Lemma \ref{nstep}, we conclude from Theorem \ref{thmcz} that $\tilde{F}_{\infty}^{N_0}$ enjoys exponential decay of correlations and CLT for dynamically H\"older continuous observables.\\

The CLT for $\tilde{F}_{\infty}$ follows easily from that of $\tilde{F}_{\infty}^{N_0}$. Now we show that exponential mixing for $\tilde{F}_{\infty}^{N_0}$ implies exponential mixing for $\tilde{F}_{\infty}$.\\
Suppose $\varphi,\phi$ are two dynamically H\"older continuous observables. For any integer $n\in\mathbb{N}$, $n = pN_0 +q$ for some integers $p>0$ and $0 \le q<N_0$.\\
We denote as $\tilde{\varphi}_q = \varphi \circ \tilde{F}_{\infty}^q$. For any $x,y$ on a same unstable manifold $W^u$, 
\[
\big| \tilde{\varphi}_q(x) - \tilde{\varphi}_q(y) \big|
= \big| \varphi(\tilde{F}_{\infty}^q x) - \varphi(\tilde{F}_{\infty}^qy) \big|
\le C \vartheta^{-q_+} \vartheta^{s_+(x,y)}
\]
where $q_+ = \min \{ q, s_+(x,y) \}$.\\
On the other hand, for any $x,y$ on a same stable manifold $W^s$, 
\[
\big| \tilde{\varphi}_q(x) - \tilde{\varphi}_q(y) \big|
= \big| \varphi(\tilde{F}_{\infty}^q x) - \varphi(\tilde{F}_{\infty}^qy) \big|
\le C \vartheta^q \vartheta^{s_-(x,y)}.
\]
Therefore $\tilde{\varphi}_q$ is also dynamically H\"older.\\
By applying the previous exponential decay result on $\tilde{F}_{\infty}^{N_0}$ with the observables $\tilde{\varphi}_q,\phi$, we know that $\exists$ $C_{\tilde{\varphi}_q,\phi}$ and $b$ such that 
\begin{align*}
\bigg| \int_{\mathbb{T}} \big( \varphi \circ \tilde{F}_{\infty}^n \big) \phi d\tilde{\mu} - \int_{\mathbb{T}} \varphi d\tilde{\mu}\int_{\mathbb{T}} \phi d\tilde{\mu} \bigg|
&= \bigg| \int_{\mathbb{T}} \big(\tilde{\varphi}_q \circ \tilde{F}_{\infty}^{pN_0} \big) \phi d\tilde{\mu} - \int_{\mathbb{T}} \tilde{\varphi}_q d\tilde{\mu}\int_{\mathbb{T}} \phi d\tilde{\mu} \bigg|\\
&\le C_{\tilde{\varphi}_q,\phi} e^{-bp} = C_{\tilde{\varphi}_q,\phi} e^{bq/N_0} (e^{-b/N_0})^n.
\end{align*}
If we take $\displaystyle C_{\varphi,\phi} = \max_q \{ C_{\tilde{\varphi}_q,\phi} e^{bq/N_0} \}$ and replace $b$ with $b/N_0$, then we have proved exponential decay of correlation in the case $\ddot{f}>0$.
\end{proof}

Next we prove the global global mixing property for the original collision map $F$.
\begin{proof}[Proof of Theorem \ref{ggm}]
Under Assumptions 1-5, the limit map $\tilde{F}_{\infty}$ satisfies the conditions of \cite{CWZ}, thus it admits a Young tower with exponential tail. We recall from Section 3.2 that $\tilde{F}_{\infty}$ well approximates the original collision map $F$ at infinity. Therefore by Theorem 2.4 and Theorem 2.9 in \cite{DN18}, $F$ is global global mixing.
\end{proof}

\subsection{Proof of Intermediary Lemmas}
In this section we prove the lemmas stated in Section 6.2.

We start with the distortion control.
\begin{proof}[Proof of Lemma \ref{distortion}]
We parametrize the unstable W as $v=\psi (t)$ for some smooth function $\psi$ such that $\psi'(t) \in [2k,2k+g/2]$ and $\psi''$ is bounded.\\
For $x,y\in W$, $\displaystyle \left| \log \mathcal{J}_W(x) - \log \mathcal{J}_W(y) \right| \le \max_{z\in W} \left| \frac{d}{dz} \log \mathcal{J}_W(z) \right| |x-y|$.\\
For $z\in W$, we take $\rm{v} = (1,\psi'(t)) \in \mathcal{T}_zW$.

{\small $$
 \mathcal{J}_W(z) = 
 \frac{\left\Vert d_z \tilde{F}_{\infty} \rm{v} \right\Vert}{\left\Vert \rm{v} \right\Vert}=$$
$$\frac{1}{\sqrt{1+\psi'(t)^2}} \left(1+4k_1^2 + \psi'(t)^2 \left( \frac{4}{g^2} +\left(1+\frac{4k_1}{g}\right)^2 \right) + 2\psi'(t) \left( \frac{2}{g}+2k_1 \left(1+\frac{4k_1}{g}\right) \right) \right)^{1/2}$$}
where $k_1 = \ddot{f}(\tilde{F}_{\infty}z)$. Hence 

{\small
 $$
    \log \mathcal{J}_W(z) = $$
    $$ \frac{1}{2} \log \left( 1+4k_1^2 + \psi'(t)^2 \left( \frac{4}{g^2} +\left(1+\frac{4k_1}{g}\right)^2 \right) + 2\psi'(t) \left( \frac{2}{g}+2k_1 \left(1+\frac{4k_1}{g}\right) \right) \right) $$
   \begin{equation}\label{logjacobian}
   - \frac{1}{2} \log (1+\psi'(t)^2)
  \end{equation}}
We note that each term inside the logarithms in \eqref{logjacobian}
is greater than one and has bounded derivatives. Thus $\left\vert \frac{d}{dz}\log \mathcal{J}_W(z) \right\vert \le C$ for some constant $C$ depending only on $\tilde{F}_{\infty}$.\\
Besides the above distortion bound, we have the following enhanced estimate.\\
Now we assume further that $W \cap \mathcal{S}^-_n = \emptyset$.\\
We denote $x_n =\tilde{F}_{\infty}^{-n} x$, $y_n = \tilde{F}_{\infty}^{-n} y$ and $W_n = \tilde{F}_{\infty}^{-n} W$.
\begin{align*}
\left\vert \log \mathcal{J}_W \tilde{F}_{\infty}^{-n}(x) - \log \mathcal{J}_W \tilde{F}_{\infty}^{-n}(y) \right\vert 
&\le \sum_{m=0}^{n-1} \left\vert \log \mathcal{J}_{W_m} \tilde{F}_{\infty}^{-n}(x_m) - \log \mathcal{J}_{W_m} \tilde{F}_{\infty}^{-n}(y_m) \right\vert \\
&\le \sum_{m=0}^{n-1} |W_m| \max_{z_m \in W_m} \left\vert \frac{d}{dz_m} \log \mathcal{J}_{W_m}\tilde{F}_{\infty}^{-1}(z_m) \right\vert.
\end{align*}
But 
\begin{align*}
\frac{d}{dz_m} \log \mathcal{J}_{W_m}\tilde{F}_{\infty}^{-1}(z_m) 
&= \frac{dz_{m+1}}{dz_m} \frac{d}{dz_{m+1}} \log \frac{1}{\mathcal{J}_{W_{m+1}}\tilde{F}_{\infty}(z_{m+1})}\\
&= - \frac{1}{\mathcal{J}_{W_{m+1}}\tilde{F}_{\infty}(z_{m+1})} \frac{d}{dz_{m+1}} \log \mathcal{J}_{W_{m+1}}\tilde{F}_{\infty}(z_{m+1}).
\end{align*}
Observe that $\mathcal{J}_{W_{m+1}}\tilde{F}_{\infty}(z_{m+1})$ is bounded. Next
$$
    \frac{dv_m}{dt_m}
       = \frac{2k_mdt_{m-1}+(4k_m/g+1)dv_{m-1}}{dt_{m-1}+\frac{2}{g}dv_{m-1}}$$
       
  $$     =\frac{2k_m+(4k_m/g+1)\frac{dv_{m-1}}{dt_{m-1}}}{1+\frac{2}{g}\frac{dv_{m-1}}{dt_{m-1}}}
       =2k_m+g/2-\frac{g/2}{1+\frac{2}{g}\frac{dv_{m-1}}{dt_{m-1}}}.
$$
Therefore 
$$\psi_m''=2\dddot{f}(t_m)-\frac{\psi_{m-1}''}{\left(1+\frac{2}{g}\psi_{m-1}'\right)^3}$$ 
which  implies that
$$|\psi_m''|\le2\dddot{f}_{\max}+\theta|\psi_{m-1}''|$$ where $\theta:=\frac{1}{(1+4k_{\min}/g)^3}<1$. 
Iterating we obtain
\begin{equation}\label{boundedcurvature}
    |\psi_m''|\le\frac{2\dddot{f}_{\max}}{1-\theta}+\theta^m|\psi_0''|.
\end{equation}
Hence $\displaystyle \left\vert \frac{d}{dz_{m+1}} \log \mathcal{J}_{W_{m+1}}\tilde{F}_{\infty}(z_{m+1}) \right\vert$ is bounded. Thus
\[
\left\vert \log \mathcal{J}_W \tilde{F}_{\infty}^{-n}(x) - \log \mathcal{J}_W \tilde{F}_{\infty}^{-n}(y) \right\vert
\le C'' \sum_{m=0}^{n-1} |W_m|
\le C'' \sum_{m=0}^{n-1} \frac{|W|}{\Lambda^m} \le C' |W|. \qedhere
\]
\end{proof}

Next we prove the complexity bound following an approach of \cite{dST}.
\begin{proof}[Proof of Lemma \ref{complexity}]
Suppose $z$ is a multiple point of $\mathcal{S}_n^+$. We take a small neighborhood of $z$ and cut it into four quadrants $Q$'s by vertical and horizontal lines through $z$. Denote as $k_n(z)|_Q$ the number of sectors cut out by $\mathcal{S}_n^+$ intersecting nontrivially with $Q$.\\
We are only interested in the active quadrants, i.e. the quadrants in the northwest and southeast, because the tangent lines to the singularities curves $\mathcal{S}_n^+$ have negative slopes and the inactive quadrants (in the northeast and southwest) remain untouched by them and thus do not contribute to the complexity growth.\\
\begin{figure}[h!]
   \centering
        \begin{tikzpicture}
            \draw (-1.5,0) -- (1.5,0)   (0,-1.3)--(0,1.3);
            \draw (0,0) circle (1cm);
            \draw[red] (-1.4,1.2) .. controls (-0.7,0.2) and (0.7,-0.2) .. (1.4,-1.2) node[anchor=west]{$\mathcal{S}_n^+$};
            \draw (0,0.2) node[anchor=west]{$z$}
                      (0.8,0.8) node[anchor=west]{$Q$};
        \end{tikzpicture}
   \caption{A multiple point and its sectors}
\end{figure}
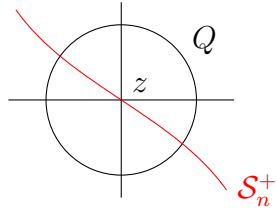

Denote as $\{ V_i \}$ the sectors cut out by $\mathcal{S}^+$. Note that $\mathcal{S}^+ = \{ \tilde{t}_1 =0\}$, hence there are at most two sectors cut out by $\mathcal{S}^+$ in a quadrant. By further cutting horizontally and vertically, we might assume each $V_i \subseteq Q$.\\
We denote as $V_i' = \tilde{F}_{\infty} (V_i)$, $z_i' = \tilde{F}_{\infty} (z_i)$ (this is defined by continuity), and $k_n(z)|_Q = \sum_i k_{n-1}(z_i')|{V_i'}$.\\
If $z \notin \mathcal{S}^+$, then $i=1$ and $k_n(z)|_Q = k_{n-1}(z')|{V'}$.\\
If $z \in \mathcal{S}^+$, then $i=2$ and we claim that at most one image $V_i'$ of the sectors $V_i$ remains active, so that in both cases we have 
\begin{align*}
k_n(z)|_Q &= \sum_i k_{n-1}(z_i')|{V_i'}\\
&\le 1+ k_{n-1}(z_i')|{V_i'} \hbox{   ($V_i'$ is the only active image)}\\
&\le 3+ k_{n-1}(z_i')|{Q_i'} \hbox{   (by further cutting $V_i'$ horizontally and vertically)}
\end{align*}
Thus $k_n(z)|_Q \le 3n$ implies $k_n(z) \le 6n$, which is our desired complexity bound.\\

Now we prove our claim. Suppose that $z\in \{ \tilde{t}_1 =0\}$.\\
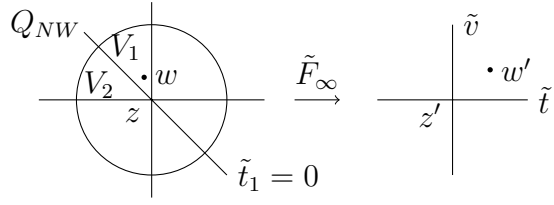
\begin{figure}[h!]
   \centering
       \begin{tikzpicture}
           \draw (-1.5,0) -- (1.5,0)   (0,-1.3)--(0,1.3)   (-0.9,0.9)--(1,-1) node[anchor=west]{$\tilde{t}_1 =0$};
           \draw (0,0) circle (1cm);
           \draw (0,-0.2) node[anchor=east]{$z$}
                     (-0.8,1) node[anchor=east]{$Q_{NW}$}
                     (-0.35,0.2) node[anchor=east]{$V_2$}
                     (0,0.7) node[anchor=east]{$V_1$};
           \draw[->] (1.9,0)--(2.5,0);
           \draw (2.2,0) node[anchor=south]{$\tilde{F}_{\infty}$};
           \draw (3,0)--(5,0) node[anchor=west]{$\tilde{t}$} 
                     (4,-1)--(4,1) node[anchor=west]{$\tilde{v}$};
           \draw (4,-0.2) node[anchor=east]{$z'$};
           \fill (4.5,0.4) circle[radius=1pt] node[anchor=west]{$w'$}
                 (-0.1,0.3) circle[radius=1pt] node[anchor=west]{$w$};
       \end{tikzpicture}
   \caption{Northwest quadrant for $z\in \{ \tilde{t}_1 =0\}$}\label{fig5}
\end{figure}

Recall that $\tilde{t}_1=\tilde{t}_0 + \frac{2\tilde{v}_0}{g}$ (mod 1), $\tilde{v}_1=\tilde{v}_0 + 2\dot{f}(\tilde{t}_1)$ (mod $g$).\\
Since $z\in \{ \tilde{t}_1 =0\}$, the $\tilde{t}$-coordinate of its images $z'$ is zero, i,e. $\tilde{t}(z_i')=0$ ($i=1,2$).\\
We pick $w\in V_1$ sufficiently close to $z$, then the $\tilde{t}$-coordinate of its image $w'$ is positive since $w$ is at the right side of the singularity line $\{ \tilde{t}_1 =0\}$. Also, since we assume $\ddot{f}>0$, $\dot{f}$ is increasing and hence the $\tilde{v}$-coordinate of its image $w'$ is larger than that of $z'$. This means that the image $V_1' = \tilde{F}_{\infty}(V_1)$ is inactive.\\
Similarly, we can show that the lower one to the left of the singularity line $\mathcal{S}^+$ in the southeast quadrant becomes inactive after being mapped by $\tilde{F}_{\infty}$.
\end{proof}

Finally we estimate the Jacobian of the holonomy map.
\begin{proof}[Proof of Lemma \ref{holonomy}]
It follows from classical results in \cite{AS,Si} that the holonomy map is absolutely continuous and its Jacobian is given by 
$$\mathcal{J}h(x)=\prod_{j=0}^{\infty}\frac{\mathcal{J}_{W_j}(x_j)}{\mathcal{J}_{\bar{W}_j}(\bar{x}_j)}$$
where $W_j/\bar{W}_j=\tilde{F}_{\infty}^j W/\tilde{F}_{\infty}^j\bar{W}$ and $x_j/\bar{x}_j=\tilde{F}_{\infty}^j x/\tilde{F}_{\infty}^j\bar{x}$.\\
As a result, 
$$\log\mathcal{J}h(x)=\sum_{j=0}^{\infty}\log\mathcal{J}_{W_j}(x_j)-\log\mathcal{J}_{\bar{W}_j}(\bar{x}_j).$$
Since $\psi'\in[2k,2k+\frac{g}{2}]$, we obtain by (\ref{logjacobian})
\begin{align*}
    &\quad\quad 2|\log\mathcal{J}_{W_j}(x_j)-\log\mathcal{J}_{\bar{W}_j}(\bar{x}_j)|\\
    &\le \bigg\vert\log\left(1+4k_{j+1}^2+\psi_j'^2\left(\frac{4}{g^2}+\left(1+\frac{4k_{j+1}}{g}\right)^2\right) + 2\psi_j' \left(\frac{2}{g}+2k_{j+1} \left(1+\frac{4k_{j+1}}{g}\right)\right)\right)\\
    &\quad-\log\left(1+4\bar{k}_{j+1}^2 + \bar{\psi}_j'^2\left(\frac{4}{g^2} +\left(1+\frac{4\bar{k}_{j+1}}{g}\right)^2\right) + 2\bar{\psi}_j' \left(\frac{2}{g}+2\bar{k}_{j+1} \left(1+\frac{4\bar{k}_{j+1}}{g}\right)\right)\right)\bigg\vert\\
    &\quad+|\log (1+\psi_j'^2)+\log (1+\bar{\psi}_j'^2)|\\
    &\le C\theta_1(|k_{j+1}-\bar{k}_{j+1}|+|\psi'_j-\bar{\psi}'_j|) + C\theta_2|\psi'_j-\bar{\psi}'_j|\\
    &\le C\theta_1(|t_{j+1}-\bar{t}_{j+1}|+|\psi'_j-\bar{\psi}'_j|) + C\theta_2|\psi'_j-\bar{\psi}'_j|
\end{align*}
where $$\theta_1^{-1/2}=1+4k_{\min}^2+4k_{\min}^2\left(\frac{4}{g^2}+\left(1+\frac{4k_{\min}}{g}\right)^2\right) + 4k_{\min} \left(\frac{2}{g}+2k_{\min} \left(1+\frac{4k_{\min}}{g}\right)\right)>1$$ 
$$\theta_2^{-1/2}=1+4k_{\min}^2>1$$
It also follows from (\ref{unstablecone}) that 
\begin{align*}
    |\psi'_j-\bar{\psi}'_j| 
    &\le C|t_j-\bar{t}_j|+C\theta_3|\psi'_{j-1}-\bar{\psi}'_{j-1}|\\
    &\le C|t_j-\bar{t}_j|+C\theta_3|t_{j-1}-\bar{t}_{j-1}|+C\theta_3^2|\psi'_{j-2}-\bar{\psi}'_{j-2}|\\
    &\cdots\\
    &\le C|t_j-\bar{t}_j|+C\theta_3|t_{j-1}-\bar{t}_{j-1}|+\cdots+C\theta_3^{j-1}|t_1-\bar{t}_1|+C\theta_3^j|\psi'_0-\bar{\psi}'_0|\\
    &\le C\frac{|t_0-\bar{t}_0|}{\Lambda^n}+C\theta_3\frac{|t_0-\bar{t}_0|}{\Lambda^{j-1}}+\cdots+C\theta_3^{j-1}\frac{|t_0-\bar{t}_0|}{\Lambda}+C\theta_3^j|\psi'_0-\bar{\psi}'_0|\\
    &\le Cj\theta_4^j|t_0-\bar{t}_0|+C\theta_3^j|\psi'_0-\bar{\psi}'_0|
\end{align*}
where $\theta_3^{-1/2}=1+4k_{\min}/g>1$, $\theta_4=\max\{\theta_3,\Lambda^{-1}\}<1$ and $\Lambda$ is the expansion rate of unstable curves.\\
Consequently, 
\begin{align}\label{ach}
    |\log\mathcal{J}_{W_j}(x_j)-\log\mathcal{J}_{\bar{W}_j}(\bar{x}_j)|\le Cj\Theta^j|t_0-\bar{t}_0|+C\Theta^j|\psi'_0-\bar{\psi}'_0|
\end{align}
where $\Theta=\max\{\theta_1,\theta_2,\theta_3,\theta_4\}<1$.\\

Finally we are ready to estimate the Jacobian. 
Observe that $s_+(x,y)=s_+(\bar{x},\bar{y})$ since each pair $(x,\bar{x})$, $(y,\bar{y})$ is connected by its corresponding stable manifold.
    $$\quad\quad|\log\mathcal{J}h(x)-\log\mathcal{J}(y)| $$
    $$\le\sum_{j=0}^{\infty}|\log\mathcal{J}_{W_j}(x_j)-\log\mathcal{J}_{\bar{W}_j}(\bar{x}_j)-\log\mathcal{J}_{W_j}(y_j)+\log\mathcal{J}_{\bar{W}_j}(\bar{y}_j)|$$
    $$\le\sum_{j<s_+(x,y)}\left(|\log\mathcal{J}_{W_j}(x_j)-\log\mathcal{J}_{W_j}(y_j)|+|\log\mathcal{J}_{\bar{W}_j}(\bar{x}_j)-\log\mathcal{J}_{\bar{W}_j}(\bar{y}_j)|\right)$$
    $$\quad +\sum_{j\ge s_+(x,y)}\left(|\log\mathcal{J}_{W_j}(x_j)-\log\mathcal{J}_{\bar{W}_j}(\bar{x}_j)|+|\log\mathcal{J}_{W_j}(y_j)-\log\mathcal{J}_{\bar{W}_j}(\bar{y}_j)|\right)$$
    $$\le C\sum_{j<s_+(x,y)}(|x_j-\bar{x}_j|+|y_j-\bar{y}_j|)+C\sum_{j\ge s_+(x,y)}j\Theta^j(|x_0-\bar{x}_0|+|y_0-\bar{y}_0|)
    $$
    $$\quad+\Theta^j(|\psi'(x_0)-\bar{\psi}'(\bar{x}_0)|+|\psi'(y_0)-\bar{\psi}'(\bar{y}_0)|)$$
    $$ \le C\Lambda^{-s_+(x,y)}(|x_{s_+(x,y)}-\bar{x}_{s_+(x,y)}|+|y_{s_+(x,y)}-\bar{y}_{s_+(x,y)}|)
    $$
    $$
    \quad+\Theta^{s_+(x,y)}(|x_0-\bar{x}_0|+|y_0-\bar{y}_0|+|\psi'(x_0)-\bar{\psi}'(\bar{x}_0)|+|
    \psi'(y_0)-\bar{\psi}'(\bar{y}_0)|)
    \le C\Theta^{s_+(x,y)}
$$
where the sum of small indices $j<s_+(x,y)$ is controlled by the distortion estimate from Lemma \ref{distortion} and the sum of large indices $j\ge s_+(x,y)$ is controlled by (\ref{ach}).
\end{proof}

\section{Escaping and Bounded Orbits}
Theorem \ref{null} shows that the escaping orbits takes up a null set. However in this section we show that the escaping orbits do exist and so do the bounded orbits.\\

We introduce the notion of \emph{proper standard pair}. A \emph{standard pair} $(W,\mu_W)$ consists of an unstable curve $W$ and a \emph{regular} probability measure $\mu_W$ supported on $W$, i.e. $\mu_W$ is absolutely continuous and has a dynamically H\"older density. We say a standard pair is \emph{proper} if there exists a large constant $C_p$ bounding the following the quantity
$$\mathcal{Z}_W:=\sup_{\epsilon}\frac{\mu_W\{r_0<\epsilon\}}{\epsilon}.$$
It is easy to see that in our case $\mu_W$ is the normalised Lebesgue measure on the unstable curve and that $\mathcal{Z}_W=\frac{2}{|W|}$, so any unstable curve $W$ longer than $\delta_2$ endowed with Lebesgue measure is a proper standard pair. We also observe that $\delta_2$ can be made arbitrarily small by choosing $C_p$ large. Therefore by Theorem \ref{exp} and Lemma 2.2, 2.3 in \cite{dn}, we have the following central limit theorem for all proper standard pairs:

\begin{prop}\label{cltcurve}
There exists $\delta_2\gg1$ such that on any unstable curve $W$ with $|W|>\delta_2$ we have the following central limit theorem for dynamically H\"older observables, i.e. 
$$\frac{1}{\sqrt{n}}\sum_{i=0}^{n-1}\varphi\circ \tilde{F}_{\infty}^i \stackrel{dist}{\rightharpoonup}\mathcal{N}(0,\sigma_{\varphi}^2)$$
where $\varphi$ is dynamically H\"older with zero average $\int_{\mathbb{T}}\varphi d\tilde{\mu}=0$.
\end{prop}

Now we prove Theorem \ref{espbdd}.
\begin{proof}[Proof of Theorem \ref{espbdd}]
Let us denote $(t_n^{\infty},v_n^{\infty})=F_{\infty}^n(t_0,v_0)$ and $(t_n,v_n)=F^n(t_0,v_0)$.\\
First we recall from Section 5 that the energy change $\gamma$ in the limit map $F_{\infty}$ on the cylinder has zero average. Moreover, $\gamma$ is dynamically H\"older as it is piecewise $C^1$ and its discontinuities are located exactly on $\mathcal{S}^+$. Therefore by Proposition \ref{cltcurve}, $\exists n_0,A$ such that for every unstable curve $W$ longer than $\delta_2$
\[
    \mathbb{P}_W\left(v_{n_0N_0}^{\infty}>v_0+A\sqrt{n_0N_0}\right)>\frac{1}{3}
\]
where $N_0$ is the constant from Lemma \ref{nstep} $N_0$-Step Expansion.\\
By Lemma \ref{growthlemma}, if $\delta_2$ is sufficiently small, then for sufficiently large $n_0$ $$\mathbb{P}_W(r_{n_0N_0}<4\delta_2)<\frac{1}{10}.$$
We know from Section 3.2 that the limit map $F_{\infty}$ well approximates the original collision map $F$ for large velocity with an error of order $\mathcal{O}(v_0^{-1})$ on each continuity component of $F_{\infty}^{n_0N_0}W$, thus we can choose $v_*\gg1$ so large that if $v_0>v_*$ everywhere on $W$, then we have $$\mathcal{P}_W(v_{n_0N_0}>v_0+A\sqrt{n_0N_0},r_{n_0N_0}>4\delta_2)>\frac{1}{4}.$$
By the estimate above, at least one component $W_1\subset F_{\infty}^{n_0N_0}W$ contains a segment $\bar{W}_1$ longer than $\delta_2$ and $v_{n_0N_0}>v_0+A\sqrt{n_0N_0}$ holds everywhere on $\bar{W}_1$. By repeating the argument on $\bar{W}_1$, we get another component $W_2\subset F_{\infty}^{n_0N_0}\bar{W}_1$ containing a segment $\bar{W}_2$ longer than $\delta_2$ and the velocity increases by another $A\sqrt{n_0N_0}$. Inductively, we construct an escaping orbit.\\
Similarly, we can construct an orbit whose velocity first increases by $A\sqrt{n_0N_0}$ and then decreases by $A\sqrt{n_0N_0}$ and etc. In that way we obtain an orbit whose energy remains in $[v_0-2A\sqrt{n_0N_0},v_0+2A\sqrt{n_0N_0}]$. This holds for arbitrarily large $v_0>v_*$, thus we have bounded orbits at arbitrarily high energy level.
\end{proof}

\section{Conclusions}
In this paper, we have studied a piecewise $C^3$-smooth Fermi-Ulam model in a constant potential field. The collision map $F$ is well approximated by the limit map $F_{\infty}$ for large velocities. $F_\infty$ covers a map $\tilde{F}_{\infty}$ on a torus. 
For admissible wall motions we proved
ergodicity, exponential decay of correlations and central limit theorem for dynamically H\"older observables.
When our assumptions fail, there are counterexamples in the class of analytic periodic platform motions by Pustylnikov \cite{Pu77} where KAM islands exist for the limit map and the original system possesses a positive measure set of escaping orbits.
The ergodic and statistical properties of the limit map $\tilde{F}_{\infty}$ established here in turn imply that the escaping set has zero measure and the typical behavior of the original collision dynamics $F$ is recurrent, but escaping and bounded orbits still exist at arbitrarily high energy level. 

It is also interesting to study long time evolution of 
 the energy distribution for typical high velocity trajectories,
 (cf. \cite{dc09, deSD} for similar results for other systems).
 Besides, we note our results do not fully address 
 the behavior of low energy orbits. The problem becomes subtle when we come to the low energy region as two consecutive collisions can happen in arbitrarily short time interval, causing the system to be nonuniformly hyperbolic. In the future, we hope to establish ergodicity for the the collision map $F$. This would imply in particular that the almost every orbit
is oscillatory, so that the energy eventually comes close to any given value.

\bibliography{FUpotential}

\begin{thebibliography}{10}

\bibitem{AS}
D.~V. Anosov and Y.~G. Sinai.
\newblock Some smooth ergodic systems.
\newblock {\em Russian Mathematical Surveys}, 22(5):103--167, 1967.

\bibitem{AZ}
M.~Arnold and V.~Zharnitsky.
\newblock Pinball dynamics: unlimited energy growth in switching {H}amiltonian
  systems.
\newblock {\em Comm. Math. Phys.}, 338(2):501--521, 2015.

\bibitem{CWZ}
J.~Chen, F.~Wang, and H.-K. Zhang.
\newblock Markov partition and thermodynamic formalism for hyperbolic systems
  with singularities, 2019.
\newblock \url{arXiv:1709.00527}.

\bibitem{Ch}
N.~Chernov.
\newblock Decay of correlations and dispersing billiards.
\newblock {\em J. Statist. Phys.}, 94(3-4):513--556, 1999.

\bibitem{dc09}
N.~Chernov and D.~Dolgopyat.
\newblock The {G}alton board: limit theorems and recurrence.
\newblock {\em J. Amer. Math. Soc.}, 22(3):821--858, 2009.

\bibitem{CZ}
N.~Chernov and H.-K. Zhang.
\newblock On statistical properties of hyperbolic systems with singularities.
\newblock {\em Journal of Statistical Physics}, 136(4):615--642, 2009.

\bibitem{DS09}
J.~de~Simoi.
\newblock Stability and instability results in a model of {F}ermi acceleration.
\newblock {\em Discrete Contin. Dyn. Syst.}, 25(3):719--750, 2009.

\bibitem{deSD}
J.~de~Simoi and D.~Dolgopyat.
\newblock Dynamics of some piecewise smooth fermi-ulam models.
\newblock {\em Chaos}, 22(2):026124, 2012.

\bibitem{dST}
J.~de~Simoi and I.~P. Toth.
\newblock An expansion estimate for dispersing planar billiards with corner
  points.
\newblock {\em Ann. Henri Poincaré}, 15(6):1223--1243, 2014.

\bibitem{Dol}
D.~Dolgopyat.
\newblock Lectures on bouncing balls.
\newblock Online notes at
  \url{http://www-users.math.umd.edu/~dolgop/BBNotes.pdf}.

\bibitem{Dol08}
D.~Dolgopyat.
\newblock Bouncing balls in non-linear potentials.
\newblock {\em Discrete Contin. Dyn. Syst.}, 22(1-2):165--182, 2008.

\bibitem{Dol-FA}
D.~Dolgopyat.
\newblock Fermi acceleration.
\newblock In {\em Geometric and probabilistic structures in dynamics}, volume
  469 of {\em Contemp. Math.}, pages 149--166. Amer. Math. Soc., Providence,
  RI, 2008.

\bibitem{dn}
D.~Dolgopyat and P.~N\'{a}ndori.
\newblock Nonequilibrium density profiles in {L}orentz tubes with thermostated
  boundaries.
\newblock {\em Comm. Pure Appl. Math.}, 69(4):649--692, 2016.

\bibitem{DN18}
D.~Dolgopyat and P.~N\'andori.
\newblock Infinite measure mixing for some mechanical systems, 2018.
\newblock \url{arXiv:1812.01174}.

\bibitem{Fermi}
E.~Fermi.
\newblock On the origin of the cosmic radiation.
\newblock {\em Phys. Rev.}, 75:1169--1174, Apr 1949.

\bibitem{GRKTChaos}
V.~Gelfreich, V.~Rom-Kedar, and D.~Turaev.
\newblock Fermi acceleration and adiabatic invariants for non-autonomous
  billiards.
\newblock {\em Chaos}, 22(3):033116, 21, 2012.

\bibitem{LaLe}
S.~Laederich and M.~Levi.
\newblock Invariant curves and time-dependent potentials.
\newblock {\em Ergodic Theory and Dynamical Systems}, 11(2):365--378, 1991.

\bibitem{LY}
M.~Levi and J-G You.
\newblock Oscillatory escape in a duffing equation with a polynomial potential.
\newblock {\em J. Differential Equations}, 140(2):415--426, 1997.

\bibitem{LiLi}
A.~J. Lichtenberg, M.~A. Lieberman, and R.~H. Cohen.
\newblock Fermi acceleration revisited.
\newblock {\em Phys. D}, 1(3):291--305, 1980.

\bibitem{LW}
C.~Liverani and M.~Wojtkowski.
\newblock Ergodicity in hamilto-nian systems.
\newblock In {\em Dynamics Reported: Expositions in Dynamical Systems}, pages
  130--202. Berlin, Heidelberg: Springer, 1995.

\bibitem{Or99}
R.~Ortega.
\newblock Boundedness in a piecewise linear oscillator and a variant of the
  small twist theorem.
\newblock {\em Proc. London Math. Soc. (3)}, 79(2):381--413, 1999.

\bibitem{Or02}
R.~Ortega.
\newblock Dynamics of a forced oscillator having an obstacle.
\newblock In {\em Variational and topological methods in the study of nonlinear
  phenomena ({P}isa, 2000)}, volume~49 of {\em Progr. Nonlinear Differential
  Equations Appl.}, pages 75--87. Birkh\"{a}user Boston, Boston, MA, 2002.

\bibitem{Pes}
Y.~B. Pesin.
\newblock Dynamical systems with generalized hyperbolic attractors: hyperbolic,
  ergodic and topological properties.
\newblock {\em Ergodic Theory and Dynamical Systems}, 12(1):123–151, 1992.

\bibitem{Pu77}
L.~D. Pustylnikov.
\newblock Stable and oscillating motions in nonautonomous dynamical systems.
  {II}.
\newblock {\em Trudy Moskov. Mat. Ob\v{s}\v{c}.}, 34:3--103, 1977.

\bibitem{Pu83}
L.~D. Pustylnikov.
\newblock On {U}lam's problem.
\newblock {\em Theoret. and Math. Phys.}, 57:1035--1038, 1983.

\bibitem{Pu94}
L.~D. Pustylnikov.
\newblock Existence of invariant curves for maps close to degenerate maps, and
  a~solution of the~{F}ermi--{U}lam problem.
\newblock {\em Mat. Sb.}, 185:113--124, 1994.

\bibitem{Si}
Y.~G. Sinai.
\newblock Dynamical systems with elastic reflections.
\newblock {\em Russian Mathematical Surveys}, 25(2):137--189, 1970.

\bibitem{Ulam}
S.M. Ulam.
\newblock On some statistical properties of dynamical systems.
\newblock In {\em Proceedings of the Fourth Berkeley Symposium on Mathematical
  Statistics and Probability, Volume 3: Contributions to Astronomy,
  Meteorology, and Physics}, pages 315--320, Berkeley, Calif., 1961. University
  of California Press.

\bibitem{Zhar}
V.~Zharnitsky.
\newblock Instability in {F}ermi-{U}lam ping-pong problem.
\newblock {\em Nonlinearity}, 11(6):1481--1487, 1998.

\end{thebibliography}
\bibliographystyle{plain}

\end{document}